\documentclass[11pt,leqno]{amsart}
\usepackage{latexsym, amsmath, amssymb, esint}
\usepackage{bm}
\usepackage{graphics,epsfig,graphicx,graphics}
\usepackage{color, xcolor}
\usepackage{tikz}
\usepackage[hyperpageref]{backref}
\usetikzlibrary{patterns}
\usepackage[english]{babel}
\usepackage[colorlinks=true,urlcolor=blue, citecolor=red,linkcolor=blue,
linktocpage,pdfpagelabels, bookmarksnumbered,bookmarksopen]{hyperref}

\newtheorem{theorem}{Theorem}[section]
\newtheorem{corollary}[theorem]{Corollary}
\newtheorem{lemma}[theorem]{Lemma}

\theoremstyle{definition}

\newcommand{\diver}{{\rm div \,}}

\numberwithin{equation}{section}

\begin{document}
    \title[]{On Serrin's overdetermined problem in space forms}

  \date{}
%
%
%
%
%
\author{Giulio Ciraolo and Luigi Vezzoni}

 \address{Giulio Ciraolo \\ Dipartimento di Matematica e Informatica \\ Universit\`a di Palermo\\ Via Archirafi 34\\ 90123 Palermo\\ Italy}
\email{giulio.ciraolo@unipa.it}

\address{Luigi Vezzoni \\ Dipartimento di Matematica G. Peano \\ Universit\`a di Torino\\
Via Carlo Alberto 10\\
10123 Torino\\ Italy}
 \email{luigi.vezzoni@unito.it}

\thanks{This work was partially supported by the project FIRB \lq\lq {\em Differential Geometry and Geometric functions theory}\rq\rq and FIR \lq\lq {\em Geometrical and Qualitative aspects of PDE}\rq\rq, and by GNSAGA and GNAMPA (INdAM) of Italy.\\
}


\begin{abstract}
We consider an overdetermined Serrin's type problem in space forms and we generalize Weinberger's proof in \cite{We} by introducing a suitable P-function.	
\end{abstract}

\maketitle

\section{Introduction}
In the seminal paper \cite{serrin} Serrin proved that if there exists a solution to 
\begin{equation} \label{eq_v_Eucl}
\Delta v + f(v) = 0
\end{equation}
in a bounded domain $\Omega \subset \mathbb{R}^n$ such that 
$$
v=0 \quad \text {and} \quad  v_\nu=const \quad \text{ on } \partial \Omega\,,
$$
then $\Omega$ must be a ball and $v$ radially symmetric. The proof in \cite{serrin} makes use of the method of moving planes and actually applies to more generally uniformly elliptic operators (see \cite{serrin}). 

In \cite{We} Weinberger considerably simplified the proof of Serrin's result in the case $\Delta v = -1$
by considering what is nowadays  called P-function and using some integral identities. The approach of Weinberger, as well as the use of a P-function, inspired several works in the context of elliptic partial differential equations (see e.g. \cite{CGS,FK,FV,FGK,GL,Payne,ronco,sperb} and references therein). 

In the present paper we investigate such symmetry results from a broader perspective of the ambient space. We consider overdetermined problems in space forms by assuming the ambient space to be a complete simply-connected Riemannian manifold with constant sectional curvature $K$.  Up to homoteties we may assume $K=0,-1,1$; the case $K=0$ corresponds to the case of the Euclidean space, $K=-1$ is the Hyperbolic space and $K=1$ is the unitary sphere with the round metric.  

In space forms, Serrin's symmetry result was proved in \cite{KP} and \cite{Mo} by adapting the proof of Serrin \cite{serrin}, i.e. by using the method of moving planes. The aim of the present paper is to prove Serrin's result in space forms by using an approach analogous to the one of Weinberger by using a suitable P-function associated to the equation $\Delta v + nKv = -1$. As in the Euclidean case, our approach is suitable only for the equation that we are considering, and does not fit with more general equations of the form \eqref{eq_v_Eucl}.

Let $(M,g)$ be a Riemannian manifold isometric to one of the following three models: the Eucliden space $\mathbb R^n$,  the Hyperbolic space $\mathbb H^n$, the hemisphere $\mathbb S^n_+$. The three models can be described as the warped product space $M=I\times \mathbb S^{n-1}$ equipped with the rotationally symmetric metric 
$$
g=dr^2+h^2\,g_{\mathbb S^{n-1}},
$$  
where $g_{\mathbb S^{n-1}}$ is the round metric on the $(n-1)$-dimensional sphere $\mathbb S^{n-1}$ and 
\begin{enumerate}
\item[-] $I=[0,\infty)$ and $h(r)=r$ in the Euclidean case ($K=0$);

\vspace{0.1cm}
\item[-] $I=[0,\infty)$ and $h(r)=\sinh(r)$ in the hyperbolic case ($K=-1$);

\vspace{0.1cm}
\item[-] $I=[0,\pi/2)$ and $h(r)=\sin(r)$ in the spherical case ($K=1$). 
\end{enumerate}

Our main result is the following. 

\begin{theorem} \label{thm_main}
Let $\Omega\subset M$ be a bounded connected domain with boundary $\partial \Omega$ of class $C^1$. Let $v$ be the solution to 
\begin{equation} \label{problem1v}
\begin{cases}
\Delta v+nK v = -1 & \text{in } \Omega \,,\\
v=0 & \text{on } \partial \Omega \,,
\end{cases}
\end{equation}
and assume that 
\begin{equation} \label{v_nu_overdet}
|\nabla v|=c \quad  \text{on } \partial \Omega \,,
\end{equation}
for some postive constant $c$. 
Then $\Omega$ is a geodesic ball $B_R$ and $v$ depends only on the distance from the center of $B_R$. 
\end{theorem}

In theorem \ref{thm_main} we may assume, up to isometries, that $B_R$ is centered at the origin. In this case $v$ is given by 
$$
v(r)=\frac{H(R)-H(r)}{n \dot h(R)} \,,
$$
with $H=\int_0^r h(s) ds$. Indeed, since the Laplacian of a radial  
radial function $u=u(r)$ is given by $\Delta u = \ddot u + (n-1)\dot h h^{-1} \dot u$, a straightforward computation yields that  $v$ solves \eqref{problem1v}. Furthermore, by computing the first derivative of $v$, we deduce that $c$ and $R$ are related by 
\begin{equation*}
c=\frac{h(R)}{n\dot h(R)} \,.
\end{equation*}

We notice that for $K=0$, \eqref{problem1v} reduces to the classical model problem $\Delta v = -1$ in the Euclidean space. The extra term $nKv$ is the one needed to obtain that the Hessian of the solution in the radial case is proportional to the metric. Moreover, this allows us to consider the $P$-function
\begin{equation} \label{Pfunction}
P(v)=|\nabla v|^2 + \frac{2}{n} v + K v^2\,,
\end{equation}
which is subharmonic when $v$ solves \eqref{problem1v}. 

Equation \eqref{problem1v} arises from the study of constant mean curvature hypersurfaces in space forms. Indeed, it is known from Reilly's paper \cite{reilly1} that a possible approach 
to prove Alexandrov {\em soap bubble theorem} in the Euclidean space is by considering the torsion potential, i.e. the solution to $\Delta v =-1$, and apply Reilly's identity. In space forms this approach was generalized by Qui and Xia in \cite{QuiXia} by replacing equation $\Delta v = -1$ with $\Delta v + nKv = -1$. 

We also mention that Alexandrov's soap bubble theorem in the Euclidean space can be proved via Serrin's overdetermined problem for the equation $\Delta v= -1$ (see \cite{reilly1}[remark at p. 468]). Hence, theorem \ref{thm_main} can be used to give an alternative proof to Alexandrov theorem in space forms by using the generalization of Reilly's identity in \cite{QuiXia}.

\medskip 
In the next section  we write $\nabla^2$ to denote the Hessian of a function and, for $X, Y$ vector fields, we write $X\cdot Y$ instead of $g(X,Y)$. 
\section{Proof of the result}

%
%
%


We first prove that the $P$-function \eqref{Pfunction} is subharmonic.
\begin{lemma} \label{lemma_P_func}
Let $v$ be a solution to
$$
\Delta v+nK v = -1
$$
and let $P$ be given by \eqref{Pfunction}.
Then 
$$
\Delta P(v) \geq 0 \quad  \text{ in } \Omega \,.
$$ 
Moreover, $\Delta P(v)=0$ if and only if 
\begin{equation}\label{nabla2v_eq}
\nabla^2 v = - \left( \frac 1n + K v\right) g \,.
\end{equation}

\end{lemma}

\begin{proof}
From the Bochner-Weitzenb\"ock formula in space forms
\begin{equation*}
\frac12 \Delta|\nabla v|^2= |\nabla^2v|^2+D(\Delta v)\cdot \nabla v+(n-1)K\nabla v\cdot \nabla v
\end{equation*}
and from Cauchy-Schwartz inequality we obtain that 
\begin{equation*}
\frac12 \Delta|\nabla v|^2\geq \frac 1n (\Delta v)^2+D(\Delta v)\cdot \nabla v+(n-1)K\nabla v\cdot \nabla v \,.
\end{equation*}
From \eqref{problem1v} we find that 
\begin{equation*}
\begin{aligned}
\frac12 \Delta|\nabla v|^2 \geq &\, \frac{1}{n}(\Delta v)(-1-nKv)+D(-1-nKv)\cdot \nabla v+(n-1)K\nabla v\cdot \nabla v\\
=&\,-\frac{1}{n}\Delta v-K v\Delta v-K \nabla v\cdot \nabla v\\
=&\,-\frac{1}{n}\Delta v-\frac{K}{2}\Delta v^2\,,
\end{aligned}
\end{equation*}
where in the last inequality we have used that $\diver(v^2/2)=v\Delta v + |\nabla v|^2$. Hence $\Delta P(v) \geq 0$.

From the argument above, we readily see that $\Delta P(v)=0$ if and only if 
$$
n |\nabla^2 v|^2 = (\Delta v)^2 \,,
$$
which implies that $\nabla^2 v$ is a multiple of the metric $g$. Since $v$ satisfies \eqref{problem1v} then we obtain \eqref{nabla2v_eq}.
\end{proof}

Lemma \ref{lemma_P_func} will be used in the following form.

\begin{corollary} \label{coroll_Pfunc}
Let $v$ be the solution to \eqref{problem1v} and assume that \eqref{nabla2v_eq} holds. Then either
\begin{equation} \label{P_constant}
P(v)=c^2 \quad \text{in } \bar \Omega
\end{equation}
or 
\begin{equation} \label{minore_stretto}
c^2 \int_\Omega \dot h  >  \left(1 + \frac 2n \right)\left(  \int_\Omega \dot h v  - K \int_\Omega  h v v_r \right) \,.
\end{equation}
\end{corollary}

\begin{proof}
From lemma \ref{lemma_P_func} we have that $\Delta P(v) \geq 0$. Since $P(v)=c^2$ on $\partial \Omega$, by the strong maximum principle either $P(v) = c^2 $ in $\bar \Omega$ or $P(v) < c^2 $ in $\Omega$. If we assume that $P(v)<c^2$ in $\Omega$ then, being $\dot h>0$, 
\begin{equation*}
c^2 \int_\Omega \dot h  >  \int_\Omega \dot h |\nabla v|^2 + \frac{2}{n} \int_\Omega \dot h v + K \int_\Omega \dot h v^2 
\end{equation*}
and since
\begin{equation} \label{div_hvDv}
\diver(\dot h v \nabla v)= \dot h |\nabla v|^2 + \dot h v \Delta v + \ddot h v v_r  
\end{equation}
and $v=0$ on $\partial \Omega$, being $\ddot h = -K h$, we obtain that 
\begin{equation*}
\begin{split}
c^2 \int_\Omega \dot h  & >  - \int_\Omega \dot h v \Delta v - \int_\Omega \ddot h v v_r + \frac{2}{n} \int_\Omega \dot h v + K \int_\Omega \dot h v^2  \\  & = (n+1)K \int_\Omega \dot h v^2 +  \left(1 + \frac 2n \right) \int_\Omega \dot h v  + K \int_\Omega  h v v_r \,,
\end{split}
\end{equation*}
which completes the proof.
\end{proof}

As we will see, \eqref{minore_stretto} will give a contradiction which follows from Poho\v{z}aev inequality.

\begin{lemma} \label{lemma_poho}
Let $v$ be the solution to \eqref{problem1v} and assume that $v$ satisfies \eqref{nabla2v_eq}. Then
\begin{equation} \label{uguale}
c^2 \int_\Omega \dot h  =  \left(1 + \frac 2n \right)\left(  \int_\Omega \dot h v  - K \int_\Omega  h v v_r \right) \,.
\end{equation}
\end{lemma}

\begin{proof}
We first consider a generic sufficiently smooth function $v$ (not necessarily a solution to \eqref{problem1v}).  We consider the Poho\v{z}aev identity in space forms (see e.g. \cite{CV})
\begin{equation}\label{Poho_identity}
{\rm div} \left(\frac{|\nabla v|^2}{2}X-hv_r \nabla v\right)=\frac{n-2}{2}\dot h|\nabla v|^2-hv_r\Delta v \,,
\end{equation}
where $X$ is the radial vector field 
$$
X = h \partial_r.
$$
Since
$$
X \cdot \nabla v = \diver (v X ) - n \dot h v \,,
$$
we have that
\begin{multline*}
\frac 1n \diver \left(|\nabla v|^2X-2(X \cdot \nabla v) \nabla v\right) - \frac{n-2}{n}\left(\diver( \dot h v \nabla v) - \dot h v \Delta v + K v X \cdot \nabla v \right) \\+ \frac 2n ( \diver (v X ) - n \dot h v)\Delta v = 0 \,,
\end{multline*}
which we write as
\begin{multline*}
\frac 1n \diver \left(|\nabla v|^2X-2(X \cdot \nabla v) \nabla v\right) \\  
- \frac{n-2}{n}\left(\diver( \dot h v \nabla v) - \dot h v (\Delta v + nKv) + n K \dot h v^2 + K v X \cdot \nabla v \right) \\ 
+ \frac 2n ( \diver (v X ) - n \dot h v)(\Delta v + n Kv) - 2 ( \diver (v X ) - n \dot h v) Kv= 0\,,
\end{multline*}
i.e.
\begin{multline} \label{pohoz_final}
\frac 1n \diver \left(|\nabla v|^2X-2(X \cdot \nabla v) \nabla v\right) \\  
- \frac{n-2}{n}\diver( \dot h v \nabla v) - \frac{n+2}{n} \dot h v (\Delta v + nKv) + \frac 2n  (\diver (v X )) (\Delta v + n Kv) \\
+( n+2)  K \dot h v^2 - \frac{n-2}{n} K v X \cdot \nabla v  - 2 K v  \diver (v X ) = 0\,.
\end{multline}
Now we assume that $v$ is a solution to \eqref{problem1v} satisfying \eqref{nabla2v_eq}, and we integrate \eqref{pohoz_final}
\begin{multline} 
- \frac{c^2}{n} \int_{\partial \Omega} X \cdot \nu
+ \frac{n+2}{n} \int_\Omega \dot h v 
+( n+2)  K  \int_\Omega \dot h v^2 \\- \frac{n-2}{n} K \int_\Omega v X \cdot \nabla v - 2 K \int_\Omega v  \diver (v X )  = 0\,,
\end{multline}
i.e.
\begin{equation*} 
- \frac{c^2}{n} \int_{\partial \Omega} X \cdot \nu
+ \frac{n+2}{n} \int_\Omega \dot h v 
-( n-2)  K  \int_\Omega \dot h v^2 + \left(\frac{2}{n}-3\right) K \int_\Omega v X \cdot \nabla v  = 0\,,
\end{equation*}
and since $\diver X=n \dot h$ we obtain
\begin{equation*} 
c^2 \int_\Omega \dot h =  \frac{n+2}{n} \int_\Omega \dot h v 
-( n-2)  K  \int_\Omega \dot h v^2 + \left(\frac{2}{n}-3\right) K \int_\Omega v X \cdot \nabla v  \,.
\end{equation*}
From \eqref{div_hvDv} we obtain \eqref{uguale}.
\end{proof}

The following result is known but, beside the Euclidean case  (see \cite{reilly2}[Lemma 3]), we didn't find a proof in letterature and we provide 
a detailed proof for reader's convenience.

\begin{lemma} \label{lemma_Obata}
Let $\Omega$ be a bounded connected domain in $M$ and assume that there exists a function $v\colon \bar \Omega\to \mathbb R$, with $v \in C^1(\overline \Omega) \cap C^2(\Omega)$, such that 
\begin{equation} \label{luigifava}
\begin{cases}
\nabla^2 v = (-\frac{1}{n}-Kv) g & \text{in } \Omega \,,\\
v=0 & \text{on } \partial \Omega \,.
\end{cases}
\end{equation}
Then $\Omega$ is a geodesic ball $B_R$ and $v$ depends only on  the center of $B_R$. 
\end{lemma}
\begin{proof}
We first notice that $v>0$ in $\Omega$. This follows from the standard maximum principles when $K=0,-1$ (see e.g. \cite{sperb}[theorem 2.6]). Now we consider the case $K=1$. We recall that the first eigenvalue of the Dirichlet Laplacian on the hemisphere is $n$ and the corresponding eigenfunction $\phi$ is strictly positive. By writing $v=w \phi$ we see that $w$ satisfies
$$
\begin{cases}
\Delta w + 2 \frac{\nabla \phi}{\phi} \cdot \nabla w < 0 & \text{in } \Omega \\
w= 0 & \text{on } \partial \Omega \,,
\end{cases}
$$
which implies that $w>0$ in $\Omega$ again by \cite{sperb}[Theorem 2.6]. Hence $v>0$ in $\Omega$.

Since $v>0$ it achieves the maximum at a point $p \in \Omega$, with $v(p)=a>0$.  Let $\gamma\colon I \to M$ be a unit speed maximal geodesic satisfying
$\gamma(0)=p$ and let $f(s)=v(\gamma(s))$.  
From \eqref{luigifava} it follows 
$$
\ddot f(s)= -\frac{1}{n}-K f(s)\,,\quad \dot f(0)=0\,,\quad f(0)=a \,,
$$
and therefore
$$
f(s)=\left(a-\frac 1n\right) H(s)-\frac 1n \,.
$$
This implies that $v$ has the same expression along any geodesic starting from $p$, and hence $v$ depends only on the distance from $p$, which completes the proof.
\end{proof}

\begin{proof}[Proof of Theorem $\ref{thm_main}$] 
Corollary \ref{coroll_Pfunc} and Lemma \ref{lemma_poho} imply that $P(v)=c^2$ and from Lemma \ref{lemma_P_func} we find that $v$ satisfies \eqref{nabla2v_eq}. Lemma \ref{lemma_Obata} gives that $\Omega$ is a geodesic ball $B_R$ and $v$ depends only on the distance from the center of $B_R$. 
\end{proof}

\end{document}